\documentclass[12pt,a4paper]{amsart}
\usepackage{graphicx,amscd,amssymb,latexsym,verbatim,hyperref}

\theoremstyle{plain}
\newtheorem{case}{Case}
\newtheorem{thm}[subsection]{Theorem}
\newtheorem{lem}[subsection]{Lemma}
\newtheorem{prop}[subsection]{Proposition} 
\newtheorem{cor}[subsection]{Corollary}
\newtheorem{conj}[subsection]{Conjecture}

\theoremstyle{definition}
\newtheorem*{ack}{Acknowledgment}
\newtheorem{properties}[subsection]{Properties}
\newtheorem{definition}[subsection]{Definition} 
\newtheorem{example}[subsection]{Example}
\newtheorem{question}[subsection]{Question}
\newtheorem{problem}[subsection]{Problem}

\newtheorem{notation}[subsection]{Notation}

\theoremstyle{remark}
\newtheorem{rem}[subsection]{Remark}

{%
 \renewcommand{\theenumi}{\alph{enumi}}%
 \renewcommand{\labelenumi}{(\theenumi)}%
 \begin{enumerate}
}%
{%
 \end{enumerate}%
}

{%
 \renewcommand{\theenumi}{\Alph{enumi}}%
 \renewcommand{\labelenumi}{(\theenumi)}%
 \begin{enumerate}
}%
{%
 \end{enumerate}%
}

{%
 \renewcommand{\theenumi}{\Alph{enumi}$^{\prime}$}%
 \renewcommand{\labelenumi}{(\theenumi)}%
 \begin{enumerate}
}%
{%
 \end{enumerate}%
}

{%
 \renewcommand{\theenumi}{\roman{enumi}}%
 \renewcommand{\labelenumi}{(\theenumi)}%
 \begin{enumerate}
}%
{%
 \end{enumerate}%
}

{%
 \renewcommand{\theenumi}{\Roman{enumi}}%
 \renewcommand{\labelenumi}{(\theenumi)}%
 \begin{enumerate}
}%
{%
 \end{enumerate}%
}

\newcommand{\one}{\mathbf{1}}
\newcommand{\A}{\mathcal{A}}

\newcommand{\FF}{\mathcal{F}}

\newcommand{\PP}{\mathcal{P}}
\newcommand{\QQ}{\mathcal{Q}}
\renewcommand{\H}{\mathcal{H}}

\newcommand{\aA}{\mathbb{A}}
\newcommand{\bB}{\mathbb{B}}
\newcommand{\gG}{\mathbb{G}}
\newcommand{\dD}{\mathbb{D}}
\newcommand{\hH}{\mathbb{H}}
\newcommand{\C}{\mathbb{C}}
\newcommand{\E}{\mathbb{E}}
\newcommand{\I}{\mathbb{I}}
\newcommand{\N}{\mathbb{N}}
\newcommand{\F}{\mathbb{F}}
\newcommand{\Q}{\mathbb{Q}}
\newcommand{\QP}{\mathcal{QP}}

\newcommand{\Z}{\mathbb{Z}}
\newcommand{\K}{\mathbb{K}}
\newcommand{\R}{\mathbb{R}}

\renewcommand{\P}{\mathbb{P}}

\newcommand{\T}{\mathbb{T}}
\newcommand{\D}{\Delta}
\newcommand{\e}{\varepsilon}
\newcommand{\g}{{\gamma}}
\newcommand{\G}{\Gamma}

\newcommand{\p}{\psi}
\newcommand{\s}{\sigma}

\renewcommand{\a}{{\alpha }}
\renewcommand{\b}{{\beta }}

\renewcommand{\l}{\lambda}

\DeclareMathOperator{\rank}{rank}
\DeclareMathOperator{\coker}{coker}

\DeclareMathOperator{\im}{im}

\DeclareMathOperator{\Hom}{Hom}

\DeclareMathOperator{\GL}{GL}

\DeclareMathOperator{\Mat}{Mat}
\DeclareMathOperator{\Char}{\operatorname{Char}}
\newcommand{\VV}{\operatorname{Char}}
\DeclareMathOperator{\car}{\operatorname{char}}

\DeclareMathOperator{\Tors}{\operatorname{Tors}}

\DeclareMathOperator{\Sh}{\operatorname{Shd}}

\newcommand\enet[1]{\renewcommand\theenumi{#1}
\renewcommand\labelenumi{\theenumi}}

\newcommand\rightmap[1]{\smash{\mathop{\rightarrow}\limits^{#1}}}

\begin{document}

\title[Quasi-projectivity, Artin-Tits Groups, and Pencil Maps]%
{Quasi-projectivity, Artin-Tits Groups, and Pencil Maps}

\author[E.~Artal]
{Enrique~Artal~Bartolo}
\address{
Departamento de Matematicas, IUMA
Universidad de Zaragoza
C/ Pedro Cerbuna, 12
50009 Zaragoza, Spain}
\email{artal@unizar.es}
\urladdr{http://riemann.unizar.es/geotop/WebGeoTo/Profes/eartal/}

\author[J.I.~Cogolludo]
{Jos\'e~Ignacio~Cogolludo-Agust{\'\i}n}
\address{Departamento de Matematicas, IUMA
Universidad de Zaragoza
C/ Pedro Cerbuna, 12
50009 Zaragoza, Spain}
\email{jicogo@unizar.es}
\urladdr{http://riemann.unizar.es/geotop/WebGeoTo/Profes/jicogo/}

\author[D.~Matei]{Daniel~Matei}
\address{
Institute of Mathematics of the Romanian Academy 
P.O. Box 1-764, RO-014700, Bucharest, Romania} 
\email{Daniel.Matei@imar.ro}
\urladdr{http://www.imar.ro/\~{}dmatei}
\thanks{The first two authors are partially supported by
MTM2007-67908-C02-01. The third author is partially supported
by SB2004-0181 and by grants 2CEx06-11-20/2006 and 
CNCSIS PNII-IDEI 1189/2008.}

%\subjclass[2000]{
%Primary 
%Secondary
%}

\keywords{Fundamental group, algebraic variety, quasi-projective group, 
pencil of hypersurfaces}

\date{\today}

\begin{abstract}
We consider the problem of deciding if a group is the fundamental group of a 
smooth connected complex quasi-projective (or projective) variety using Alexander-based 
invariants. In particular, we solve the problem for large families of Artin-Tits groups.
We also study finiteness properties of such groups and exhibit examples of hyperplane 
complements whose fundamental groups satisfy $\text{F}_{k-1}$ but not $\text{F}_k$ for any $k$. 
\end{abstract}

\maketitle

\begin{center}
\emph{To Anatoly Libgober}
\end{center} 

\section{Introduction}

The interest in characterizing or finding properties of (quasi-)projective groups,
that is, groups that can be obtained as fundamental groups of (quasi-)projective
varieties, has been known since J.-P.~Serre~\cite{Se} (who raised questions about 
general characterization of such groups) and O.~Zariski~\cite{Zariski-algebraic}
(who asked whether or not they are residually finite). The latter question was 
negatively answered by D.~Toledo~\cite{Toledo}. Only recently, after the work 
published by A.~Libgober~\cite{Li5} and D.~Arapura~\cite{Ar2}, the question about 
characterization has experienced an increasing interest.

Serre's original question is far from being solved, but at least new obstructions 
have been found to be effective to show that certain groups cannot be (quasi-)projective
in a series of papers by Dimca-Suciu-Papadima~\cite{DPS1, DPS2, DPS3, DPS4}.
Such obstructions are mostly based on Alexander invariants and finiteness conditions.

In this paper, we are mainly concerned about determining which Artin groups are 
quasi-projective and which ones are not by exhibiting quasi-projective varieties that
realize them in the first case and by using obstructions in the second case.
Finally, we are also interested in finiteness conditions by means of studying certain
subgroups of quasi-projective groups that appear as kernel of pencil based maps.
The reason for considering Artin groups has to do with two facts: first of all, it is a
large class of groups that can be easily described with a decorated 1-graph, but most 
importantly, the way these groups are built is by using \emph{algebraic} relations, i.e. 
relations appearing in local fundamental groups of algebraic singularities of type $\aA_k$.

The Alexander invariant of a (quasi-)projective variety $M$ is the first homology group of the 
universal abelian cover considered as a module over the group of deck transformations. As it turns 
out, an important invariant of such a module is related with the space of rank-one representations
of the (quasi-)projective fundamental group $G$ of $M$. In fact one has a stratification of the 
space of local systems of rank 1 whose strata are called \emph{characteristic varieties}
of $M$ (or $G$). Characteristic varieties of quasi-projective varieties were first considered
by D.~Arapura (cf.\cite{Ar1}) who gave a structure theorem on the biggest stratum, namely, 
he proved that it is a finite union of translated subtori (by a torsion element) and a finite
union of unitary elements. This was completed by A.~Libgober~\cite{Li9} (in the local case) and 
N.~Budur~\cite{Bu} (in the global case following the work of C.~Simpson~\cite{Si1} in the 
projective case) 
showing that the unitary elements should in fact be torsion elements. This structure theorem
and new properties presented in several papers (see~\cite{Di4,DPS0,DPS0,AC-prep} among others) 
impose restrictions on $G$. Our purpose here is to show how these restrictions are enough to 
prove the non quasi-projectivity of a large family of Artin-Tits groups. Also, an infinite
number of realizable triangular Artin-Tits groups are shown. 

As for the finiteness conditions, in a series of papers~\cite{Wa1, Wa2} from 1960's C.T.C.~Wall 
studied general finiteness properties of groups and ${\rm CW}$-complexes. A group $G$ is said 
to be of type $\text{F}_n$ if it has an Eilenberg-MacLane  complex $K(G,1)$ with finite $n$-skeleton.
Clearly $G$ is finitely generated if and only if it is $\text{F}_1$ and finitely presented if and 
only if it is $\text{F}_2$. An interesting example of a finitely presented group which is not finitely 
presented was given by J.~Stallings in~\cite{St}. A group $G$ is said to be of type $\text{FP}_n$ if 
the trivial $\Z G$-module $\Z$  admits a projective resolution which is finitely generated in 
$\text{dimensions}\leq n$. Note that $G$ is of type $\text{FP}_1$ if and only if it is finitely 
generated. In general, the property $\text{F}_n$ implies the ${\rm FP}_n$ property, and they are 
equivalent in the presence of $\text{F}_2$. But, as shown by Bestvina-Brady~\cite{BB} 
$\text{FP}_2$ does not imply finite presentation. The first example of a group which is $\text{F}_2$ 
but not $\text{FP}_3$ was given by J.~Stallings in~\cite{St}. Afterwards R.~Bieri~\cite{Bi1,Bi2} 
generalized Stallings' examples to the following family: Let $G_n = \F_2 \times \cdots \times \F_2$ be 
the direct product of $n$ free groups, each of rank $2$. Then the kernel of the map taking each 
generator to $1 \in \Z$ is $\text{F}_{n-1}$ but not $\text{F}_n$. Stallings' examples mentioned 
above are the cases $n=2$ and $n=3$. Analogous, more general kernels, now known as Bestvina-Brady groups,
coming from arbitrary right-angled Artin groups, were considered in~\cite{BB}. 

In our work, we propose an approach towards an answer to Serre's
question based on the cohomology associated to rank $1$ complex
representations of $G=\pi_1(M)$ and its relationship with the
geometric properties of pencils of hypersurfaces on $M$.
This point of view will lead us naturally into considering the
finiteness properties of $G$. In this respect, we will pay particular 
attention to the groups $G$ which can be realized as 
$\pi_1(\P^r\setminus \mathcal{A})$ for $\mathcal{A}$ 
a hypersurface arrangement in  $\P^r$. 

We find that the Bieri-Stallings groups appear as $\pi_1(\P^r\setminus \mathcal{A})$ for
some arrangements of hyperplanes $\A$. In this way, we exhibit a family 
of hyperplane arrangement groups that are $\text{F}_{n-1}$ but not 
$\text{F}_n$, for any $n\ge 3$. The entire class of quasi-projective 
Bestvina-Brady groups was characterized in~\cite{DPS2}. We will refine elsewhere
that result by showing that all quasi-projective Bestvina-Brady groups
are in fact hyperplane arrangement groups. Here we only give the argument
for the Bieri-Stallings groups.

The paper is organized as follows: in the first three sections we give an
expository presentation of the objects, techniques, and tools necessary.
In particular section~\S\ref{sec-qproj} will deal with definitions and 
properties of (quasi-)projective groups, sections~\S\ref{sec-charvar} 
and~\S\ref{sec-alex} with definitions and properties of two of the most important
invariants of finitely presented groups (in this context) such as the 
Characteristic Varieties and Alexander Invariants respectively. 
In section~\S\ref{sec-artin} we will study the question of what Artin
groups are (quasi-)projective. Finally, in section~\S\ref{sec-pencilmaps} we
will study some finiteness properties of quasi-projective groups via the 
pencil map construction.

\medskip

\begin{ack}
The third author would like to thank Michael Falk for an inspiring discussion.
A substantial amount of this work was done while the third author 
was visiting Universidad de Zaragoza.
He is grateful for the financial support received from his
institution, as well as the host institution, 
that made that visit possible.
\end{ack}

\section{(Quasi)-projective groups}
\label{sec-qproj}

In this section we will describe the class of finitely presented groups
we are interested in.
It is known that any finitely presented group $G$ can be 
realized as $\pi_1(M)$ of a smooth connected complex manifold $M$. If $M$ 
is allowed to be compact, then its complex dimension cannot be in general 
less than~$4$. In terms of $\pi_1(M)$, if one allows $M$ to be open, then 
$M$ might be chosen just to be~$2$-dimensional. If in addition, $M$ is 
required to be Stein, then again its dimension cannot be less than~$4$.

However, if $M$ is required to be a smooth complex algebraic variety, then 
not any finitely presented $G$ can be realized as $\pi_1(M)$. In the compact
case there is a well-known restriction, namely $G$ must be $1$-formal 
(cf.~\cite{DGMS}). Rational homotopy theory imposes restrictions on $\pi_1(M)$ 
even in the non-compact case, where $M$ is quasi-projective, as shown 
in~\cite{Mo}, but $1$-formality is no longer guaranteed.

\begin{definition}
We call a group $G$ \emph{projective} (resp. \emph{quasi-projective})
if $G=\pi_1(M)$ for $M$ a smooth connected complex projective
(resp. quasi-projective) variety. 
\end{definition}

We denote by $\QP$ the set of all quasi-projective groups, and by $\PP$
the set of all projective groups. Obviously we have an inclusion,
$\PP\subset\QP$, which is in fact strict, since groups $G$ with odd first
Betti number $b_1(G)$ cannot be projective, whereas for instance, free groups 
of odd rank are clearly quasi-projective. 

\begin{rem}
Standard Lefschetz-Zariski-Van Kampen principle guarantees that all groups in 
$\QP$ are finitely presented. Moreover, any (quasi-)projective group is the 
fundamental group of a (quasi-)projective surface.
\end{rem}

Let $\FF_1$ stand for the set of all finitely presented $1$-formal groups. 
We pointed out above that there is an inclusion $\PP\subset\FF_1$. This inclusion 
is also strict by an easy odd first Betti number argument. However, in general, 
quasi-projectivity and $1$-formality are independent properties (cf.~\cite{DPS0, DPS1}). 
An important class of groups which are both $1$-formal and quasi-projective is that of 
fundamental groups of complements to hypersurfaces in a projective space. We denote by 
$\mathcal{H}$ the set of such groups, that is, $G=\pi_1(M)$ where $M=\P^N\setminus V$ 
for a hypersurface $V$ in $\P^N$; we call these groups \emph{hypersurface} or \emph{curve} 
groups (again because of Lefschetz-type arguments). Then $\mathcal{H}\subset\FF_1\cap\QP$, 
as shown in~\cite{Ko2} using techniques from~\cite{Mo}. 

\begin{properties}
The following properties have been proved in the literature:

\begin{enumerate}
\enet{($\PP$\arabic{enumi})}
\item  If $G_1, G_2$ are in $\QP$, then the direct product $G_1\times G_2$ is also in $\QP$.
\item  If $G_1, G_2$ are in $\FF_1$, then the direct product $G_1\times G_2$ is also in $\FF_1$.
\item  The free product $G_1*G_2$ is not necessarily in $\QP$ (cf.~\cite{DPS0, DPS1}). 
For example, $\pi_1*\pi_2$ is not quasi-projective if $\pi_1, \pi_2$ are fundamental groups 
of Riemann surfaces of non-zero genus.
\item By M.~Gromov~\cite{Gr}, a non-trivial free product $G_1*G_2$ is never projective.
\item If a group $G$ has more than one end, then $G\notin\PP$. Also if $G$ surjects with finitely 
generated kernel onto a group with infinitely many ends, then again $G\notin\PP$ (cf.~\cite{ABR}).
\item If $G_1, G_2$ are in $\FF_1$, then so is $G_1*G_2$, see~\cite{DPS0, DPS1}.
\item The classes $\PP$ and $\QP$ are closed under taking finite index subgroups. 
The same is not true in general for the class $\FF_1$, see~\cite[Example~2.9]{DP}. 
\end{enumerate}
\end{properties}

\begin{example} 
The fundamental group 
$\pi_g=\langle a_i, b_i, 1\le i\le g\mid \prod_{i=1}^{g} a_ib_ia_i^{-1}b_i^{-1}\rangle$ 
of a smooth compact projective curve $C_g$ of genus $g$ is the most basic projective group. 
Maybe the simplest infinite quasi-projective groups are the free groups, both abelian $\Z^n$,
and non-abelian $\F_n$. Of these, only $\Z^{2n}$ are projective. Note that 
$\F_{2g+p-1}=\pi_1(C_{g,p})$, where $C_{g,p}, p\ge 1$ is $C_g$ punctured $p$ times. 
\end{example}

We are going to extend these examples to \emph{orbifold groups}.

\begin{definition}
An \emph{orbifold} $X_\varphi$ is a quasi-projective Riemann surface~$X$
with a function $\varphi:X\to\N$ with value~$1$ outside a finite
number of points.
\end{definition}

\begin{definition}
\label{def-orb-morph}
For an orbifold $X_\varphi$, let $p_1,\dots,p_n$ the points such that
$\varphi(p_j):=m_j>1$. Then, the \emph{orbifold fundamental group} of $X_\varphi$ is
$$
\pi_1^{\rm\text{orb}}:=\pi_1(X\setminus\{p_1,\dots,p_n\})/\langle\mu_j^{m_j}=1\rangle
$$
where $\mu_j$ is a meridian of $p_j$ in $X$. We oftentimes denote $X_\varphi$ simply by $X_{m_1,\dots,m_n}$.
\end{definition}

\begin{definition}
A dominant algebraic morphism $\rho:Y\to X$ defines an \emph{orbifold morphism} $Y\to X_\varphi$
if for all $p\in X$, the divisor $\rho^*(p)$ is a $\varphi(p)$-multiple.
\end{definition}

\begin{rem}\label{rem-exseq}
An orbifold morphism $\rho:S\to C_\varphi$ with connected generic fiber~$F$ defines an exact sequence
\[
\pi_1(F)\to\pi_1(S)\to\pi_1^{\rm\text{orb}}(C_\varphi)\to 1. 
\]
\end{rem}

\begin{example} 
Any orbifold fundamental group $G:=\pi_1^{\text{orb}}(C_\varphi)$
can be realized as $\pi_1(S)$ for $S$ a ruled affine surface, as shown 
by J.~Bertin~\cite{Ber}. If the underlying curve is not projective, it is easily seen that
$G$ is not projective, since it admits a finite-index free subgroup.

On the other hand, if $C$ is projective, one can show that $G$ is projective, using arguments 
used by Morgan-Friedman in~\cite{MF}. They show that, for an elliptic fibration $\pi:S\to C$
with at least one fiber having singular reduction, the exact sequence in Remark~\ref{rem-exseq}
induces an isomorphism of $\pi_1(S)$ onto $G$. Such an elliptic fibration can be found using 
logarithmic transformations. Moreover, it is proved in~\cite{MF} that such a surface can be 
deformed to become projective if and only if $b_1(S)$ is even, and thus the result follows.
\end{example}

In light of the discussion above, it seems legitimate to ask the following:

\begin{question}
Which finitely presented groups $G$ are quasi-projective? 
\end{question}

This question, originally posed by J.-P.~Serre in~\cite{Se}, was taken up by A.~Libgober 
in~\cite{Li5} using Alexander invariants. 
This is the point of view that will be considered here as well.

\section{Characteristic varieties}
\label{sec-charvar}
Throughout this paper $G$ will representa a finitely presented group. 
Characteristic varieties are invariants of $G$ which can be computed using any connected 
topological space $X$ (having the homotopy type of a finite CW-complex) such that $G=\pi_1(X,x_0)$,
$x_0\in X$. Let us consider a character $\xi:G\to\C^*$; recall that the space of characters is
\begin{equation}
\label{eq-torus}
\T_G=\Hom(G,\C^*)=\Hom(H_1(X;\Z),\C^*)=H^1(X;\C^*).
\end{equation}
Given such a character $\xi$, one can construct a local system of coefficients over $X$ as follows.
Let $\rho:\tilde{X}\to X$ be the universal abelian covering of $X$. The group $H_1(X;\Z)$ acts freely 
(on the right) on $\tilde{X}$. The local system of coefficients $\C_\xi$ is defined as 
\[
\pi_\xi:\tilde{X}\times_{H_1(X;\Z)}\C\to X
\text{ where }
\tilde{X}\times_{H_1(X;\Z)}\C:=
\left(
\tilde{X}\times\C
\right)\Big/
(x,t)\sim(x^h,\xi(h^{-1}) t).
\]

\begin{definition}
\label{def-char-var}
The $k$-th~\emph{characteristic variety}~of $G$ is the subvariety 
of $\T_G$, defined by:
\[
\VV_{k}(G)=\{ \xi \in \T_G\: |\:\dim H^1(G,\C_{\xi}) \ge k \}, 
\]
where $H^1(G,\C_{\xi})$ is the twisted cohomology with coefficients 
in the local system $\xi$. 
\end{definition}

\begin{rem}\label{rem-cell}
There is a simple way to describe this cohomology. Let us suppose that $X$ is a CW-complex. Then, $\tilde{X}$ inherits also 
a CW-complex structure. For each cell $\sigma$ in $X$ we fix a cell $\tilde{\sigma}$ in $\tilde{X}$ such that
$\rho(\tilde{\sigma})=\sigma$. Then, the set of cells of $\tilde{X}$ is 
\[
\{\tilde{\sigma}^h\mid\sigma\text{ cell of }X, h\in H_1(X;\Z)\}.
 \]
In particular, the chain complex $C_*(\tilde{X};\C)$ is a free $\Lambda$-module with basis
$\{\tilde{\sigma}\}$, where $\Lambda:=\C[H_1(X;\Z)]$ is the group algebra of $H_1(X;\Z)$.
Evaluating the elements in $H_1(X;\Z)$ by $\xi$ we obtain a chain complex $C_*(X;\C)^\xi$, which as
a vector space, is isomorphic to $C_*(X;\C)$ but whose differential is twisted.
\end{rem}

\begin{rem}
Let us assume that $G$ is finitely generated; then so is $H_1(X; \Z)$. Let $n:=\rank H_1(X; \Z)$ and 
let $\Tors_G$ be the torsion subgroup of $H_1(X; \Z)$. Then $\T_G$ is an abelian complex Lie group 
with $|\Tors_G|$ connected components (each one isomorphic to $(\C^*)^n$) satisfying the following exact 
sequence:
\[
1\to \T_G^\one\to\T_G\to \Tors_G\to 1,
\]
where $\T_G^\one$ is the connected component containing the trivial character $\one$.
\end{rem}

\begin{example}\label{ex-noqp}
Let $G:=\langle x,t\mid x t^2 x=t^2 x^2 t\rangle$. In this case $\T_G=\C^*$ and
$\VV_1(G)$ is defined by $z^2-2 z+2=0$.
\end{example}

These invariants are very close to other ones, like the Green-Lazarsfeld's invariant. 
They were studied by A.~Beauville~\cite{Be} for \emph{projective surfaces} and his approach was extended 
by D.~Arapura~\cite{Ar1} to quasi-projective varieties. There are also important contributions from 
C.~Simpson~\cite{Si1}, N.~Budur~\cite{Bu}, T.~Delzant~\cite{Delzant}, and A.~Dimca~\cite{Di4} for the 
structure of characteristic varieties (for compact K\"ahler or quasi-projective manifolds). In the 
hypersurface case, A.~Libgober~\cite{Li3,Li4} proposed a computation method where the knowledge 
of the group is not required and showed that they reflect deep algebraic properties of the manifolds. 

The properties of characteristic varieties of quasi-projective groups provide strong obstructions 
for a group to be in~$\QQ\PP$. Following Remark~\ref{rem-cell}, it is not difficult to prove that the
characteristic varieties of a quasi-projective group $G$ are algebraic subvarieties of $\T_G$. 
The following result can be found in one form or another in the literature:

\begin{thm}[Arapura~\cite{Ar1}, Budur~\cite{Bu}]
\label{thm-donu-nero}
If $G \in \QP$, then all irreducible components of $\VV_k(G)$ are subtori of the character torus $\T_G$
(possibly translated by a torsion character).
\end{thm}

Note that, in particular, the group of Example~\ref{ex-noqp} cannot be quasi-projective. 

The main part of Theorem~\ref{thm-donu-nero} was proved by D.~Arapura, the only part missing in his 
result is the fact that isolated points are not only unitary, but in fact torsion characters. 
The main result of D.~Arapura involves more obstructions.

\begin{thm}[Arapura~\cite{Ar1}]
\label{thm-arapura}
Let $\Sigma$ be an irreducible component of $\VV_1(G)$. Then,
\begin{enumerate}
\enet{\rm(\arabic{enumi})}
\item\label{thm-arapura-1}
If $\dim\Sigma>0$ then there exists a surjective morphism $\rho:X\to C$,
$C$ algebraic curve, and a torsion element $\sigma$ such that
$\Sigma=\sigma\rho^*(H^1(C;\C^*))$.
\item If $\dim\Sigma=0$ then $\Sigma$ is unitary.
\end{enumerate}
\end{thm}

We give here a refined statement which will be proved in a forthcoming paper~\cite{ACM-prep}.

\begin{thm}\label{thm-orb}
Let $G\in \QP$ and $\Sigma$ be an irreducible component of $\VV_k(G)$, $k\geq 1$. Then one of 
the two following statements holds:
\begin{itemize}
\item There exists a surjective orbifold morphism $\rho:X\to C_\varphi$
and an irreducible component $\Sigma_1$ of $\VV_k(\pi_1^{\rm\text{orb}}(C_\varphi))$
such that $\Sigma=\rho^*(\Sigma_1)$.
\item $\Sigma$ is an isolated torsion point.
\end{itemize}
\end{thm}

\begin{rem}
In many cases, isolated torsion points come from orbifold morphisms. It will be proved
in a forthcoming paper~\cite{AC-prep} that this is not always the case.
\end{rem}

\begin{definition}
Let $\Sigma$ be an irreducible component of $\VV_k(G)$ and $\dim_\C \Sigma\geq 1$, consider
$\Sh \Sigma$ (not necessarily in $\VV_k(G)$) parallel to $\Sigma$ (that is, $\Sigma=\rho \Sh \Sigma$
for some $\rho\in \T_G$) and such that $\one \in \Sh \Sigma$. Such a subtorus $\Sh \Sigma$ will be 
referred to as the \emph{shadow} of~$\Sigma$.
\end{definition}

The result and the proof of Theorem~\ref{thm-orb} provide many obstructions which will be used in \S\ref{sec-artin}. 

\begin{prop}\label{prop-obs1}
Let $G\in \QP$ and $\Sigma_1,\Sigma_2$ different irreducible components of  $\VV_k(G)$, 
$k\geq 1$ of positive dimension. Then
\begin{enumerate}
\enet{\rm(\arabic{enumi})}
\item\label{prop-obs1-1} If the intersection $\Sigma_1\cap\Sigma_2$ is non-empty, it consists of isolated 
torsion points, which are of torsion type.
\item\label{prop-obs1-2} 
Their shadows are either equal or have $\one$ as an isolated intersection point.
\item\label{prop-obs1-3} 
If $\Sigma_1$ is not a component of $\VV_{k+1}(G)$ and $p\in\Sigma_1\cap\VV_{k+1}(G)$ then $p$~is a torsion point.
\item\label{prop-obs1-4} $\Sigma_1$ is an irreducible component of $\VV_\ell(G)$, $1\leq\ell\leq k$. 
\end{enumerate}
\end{prop}

\begin{rem}
Parts~\ref{prop-obs1-1} and~\ref{prop-obs1-2} in Proposition~\ref{prop-obs1} can be found in~\cite{DPS4},
Part~\ref{prop-obs1-3} is proved in~\cite{Di3}, and Part~\ref{prop-obs1-4} is immediate.
\end{rem}

\begin{prop}\label{prop-obs2}
Let $G\in \QP$ and $\Sigma$ be an irreducible component of $\VV_k(G)$, $k\geq 1$ of positive dimension~$d$.
Then:
\begin{enumerate}
\enet{\rm(\arabic{enumi})}
\item \label{prop-obs2-1} 
If $\one\in\Sigma$, then $k\leq d-1$. Moreover, one can ensure that $\Sigma$ is a component of $\VV_{d-2}(G)$ 
(resp. $\VV_{d-1}(G)$) if $d$ even (resp. odd).
\item\label{prop-obs2-2} 
If $\one\notin\Sigma$, then $\Sigma$ is a component of $\VV_d(G)$. 
\item\label{prop-obs2-3} 
If $\one\notin\Sigma$ and $d>2$, then its shadow
is an irreducible component of $\VV_1(G)$.
\item\label{prop-obs2-5} 
If $\one\notin\Sigma$ and $d=2$, then its shadow
is an irreducible component of $\VV_1(G)$ if and only if it is for $\VV_2(G)$.
\item\label{prop-obs2-6} 
If $\one\notin\Sigma$ and $d=1$, then its shadow
is not an irreducible component of $\VV_1(G)$.
\end{enumerate}
\end{prop}

\begin{rem}
The results in Proposition~\ref{prop-obs2}\ref{prop-obs2-5}, and~\ref{prop-obs2-6} can be found 
in~\cite{Di3}. All of them will appear in~\cite{AC-prep}. The cases where the shadow is not in the 
characteristic variety correspond, according to Theorem~\ref{thm-orb}, to either orbifold pencils 
over $\C^*$ or elliptic pencils.
\end{rem}

\begin{definition}
A subspace $0\neq V\subset H^1(G,\C)$ is called $0$-\emph{isotropic} (resp. $1$-\emph{isotropic}) 
if the restriction $\cup_V$ of the cup-product map $\cup_G: \bigwedge^2 H^1(G,\C)\to H^2(G,\C)$ is 
equivalent to $\cup_C: \bigwedge^2 H^1(C,\C)\to H^2(C,\C)$ for $C$ a non-compact (resp. compact) 
smooth connected complex curve (see also~\cite[Definition~6.5]{DPS0}).
\end{definition}

\begin{prop}[\cite{DPS0,DPS4}]
Let $\Sigma$ be an irreducible component of $\VV_1(G)$, $G\in\QP$. Let $V\subset H^1(G,\C)$ be the 
tangent space of $\Sh \Sigma$ at~$\one$. Then $V$ is $p$-isotropic, $p\in \{0,1\}$.
\end{prop}

We finish with a new obstruction which will appear in a forthcoming paper~\cite{ACM-prep}.

\begin{prop}
\label{prop-suma}
Let $\Sigma_1$ be an irreducible component of $\VV_k(G)$, $G\in\QP$, and let $\Sigma_2$ be an 
irreducible component of $\VV_\ell(G)$. 
If $\xi\in\Sigma_1\cap\Sigma_2$ a torsion point, then $\xi\in\VV_{k+\ell}(G)$.
\end{prop}

\section{Alexander Invariants}
\label{sec-alex}
Let $G$ be a finitely generated group, consider $H:=H_1(G; \Z)=G/G'$ and $\psi$ a 
surjective homomorphism from~$H$ onto a group $A$ (note $A$ has to be abelian). 
We are mostly interested in the case $H=A$ (as in \S\ref{sec-charvar}) but it is not more difficult
to work in this more general situation.
Note that there is a short exact sequence
$$
0 \ \rightarrow \ K \ \rightmap{\varphi} \ 
G \ \rightmap{\tilde{\psi}} \ A \ 
\rightarrow \ 0.
$$
The group $A$ acts on $M^A_G:=H_1(K; \Z)$ by conjugation.
This makes $M^A_G$ a module over the group algebra $\Lambda^A_\Z:=\Z[A]$.
Note that if $A=H$ then $K=G'$

\begin{definition}.
The $\Lambda^A_\Z$-module $M^A_G$ is the \emph{first Alexander Invariant} of
$G$ with respect to $A$.
Analogously, one can define the \emph{first Alexander Invariant} 
$M^A_{\K,G}$ of $G$ with respect to $A$ over $\K$ as a $\Lambda^A_\K$-module, where 
$\Lambda^A_\K:=\K\otimes_\Z \Lambda^A_\Z$.  We will drop the superscript $A$ when $A=H$. 
\end{definition}

This is a powerful invariant, but it is not easy to deal with it directly.
A classical object of study is the set of Fitting invariants of $M^A_G$. Consider 
$$
\left(\Lambda^A_\Z\right)^m \ \rightmap{\phi} \ \left(\Lambda^A_\Z\right)^r \ \rightarrow \ M^A_G \ \rightarrow \ 0
$$
a finite free presentation of the $\Lambda^A_\Z$-module $M^A_G$. 
Let $\Mat(\phi)$ be the $(r \times m)$ matrix, with coefficients in $\Lambda^A_\Z$, of $\phi$.

\begin{definition}
\label{deffitting}
The $k$-th \emph{Fitting ideal} of $M^A_G$ is defined as the 
ideal generated by
$$
f_k^A:=
\left\{
\array{ll}
0 & \text{if \ } k \leq \max\{0,r-m\} \\
1 & \text{if \ } k > r \\
\text{minors of } \Mat(\phi) \ \text{of order } 
(r-k+1) & \text{otherwise.}
\endarray
\right.$$
Such an ideal does not depend on the free presentation of $M^A_G$ 
and it is denoted by $F^A_k$ if no ambiguity 
seems likely to arise. Analogously, one can define $F_k(M^A_{\K,G})$.
\end{definition}

\begin{rem}\label{rem-fitting}
For computational purposes (when $A$ is a free abelian group, for simplicity), if one writes a 
presentation of $G$ of the following type 
$$
G=\langle x_1,...,x_r,y_1,...,y_s : \bar R(\bar x,\bar y)=0 \rangle,
$$
where $x_1,...,x_r$ freely generate $A$ and $y_1,...,y_s$ are trivial after the abelianization morphism, 
then $M^A_\Z$ admits a presentation 
\begin{equation}
\label{mat-phi}
R \Lambda^A_\Z \oplus J \Lambda^A_\Z \ \rightmap{\phi}\ 
\left( \bigoplus_{1\leq i<j\leq r} x_{ij} \Lambda^A_\Z \right)\bigoplus \left( \bigoplus_{k=1}^s y_k \Lambda^A_\Z \right)
\to M^A_\Z \to 0
\end{equation}
as a $\Lambda^A_\Z$-module generated by $x_{ij}:=[x_i,x_j]$, $1\leq i<j\leq r$ and $y_1,...,y_s$. 
A complete set of relations can be given by the \emph{Jacobi relations}:
$$
(t_i-1)x_{jk}+(t_j-1)x_{ki}+(t_k-1)x_{ij}=0
$$
and the rewriting of $\bar R(\bar x,\bar y)=0$ in terms of $x_{ij}$ and $y_k$. For a detailed description
of this module see~\cite[\S~2.5]{ZPairs}.
\end{rem}

The ring $\Lambda^A_\C$ as a $\C$-algebra of finite type, is the ring of functions of an
affine variety $\T^A_G$ which is a (maybe 
non-connected) complex torus (like in~\eqref{eq-torus}, replacing $H$ by $A$)

\begin{definition}\label{def-char-alex}
The reduced zero locus of $F_k(M^A_{\C,G})$ is the \emph{$k$-th characteristic variety of $G$ with respect to $A$} and 
is denoted 
by $\Char^A_k(G)$.
\end{definition}

\begin{rem}
The Definitions~\ref{def-char-alex} of $\Char^H_k(G)$ and~\ref{def-char-var} of $\Char_k(G)$
agree with the convention of forgetting the superscript (with the exception of the behaviour of $\one$), see, for instance,~\cite{Li3} or a sketch in Remark~\ref{rem-compare}\ref{rem-compare3}.
\end{rem}

In case $A$ is a free abelian group of rank $r$, then $\Z [A]=\Z[t_1^{\pm 1},...,t_r^{\pm 1}]$ is the ring of 
Laurent polynomials in $n$ variables and $\Q [A]$ is a UFD. 

\begin{definition}
For $A$ a free abelian group, the \emph{multi-variable Alexander polynomial of $G$ with respect to $A$} 
is the Laurent polynomial $\D^A_{G}$ in $\Q [A]$ defined by $\D^A_G=\gcd \,F_1(M^A_G)$.
Analogously, one has the \emph{multi-variable Alexander polynomial of $G$ with respect to $A$} over $\K$. 
\end{definition}

\begin{rem}
Note that $\D_{\K,G}$ is only well defined up to multiplication by a unit in $\Lambda_\K$, and hence any
equality involving $\D_{\K,G}$ has to be considered up to a unit.
\end{rem}

\begin{example}[\cite{ZPairs}]
Consider the group $G:=\langle a,b \mid aba=bab, [a,a^2b^2]=[b,a^2b^2]=1 \rangle$ and the 
abelianization morphism $G\to \Z$, then 
$$
M_{\K,G}=
\begin{cases} 
\frac{\K[t^{\pm 1}]}{(t+1)} & \text{ if } \car(\K)=3 \\ 0 & \text{ otherwise.}
\end{cases}
$$
Therefore 
$$\Char_k(\K,G)=
\begin{cases} 
\{-1\}\subset \K^* & \text{ if } \car(\K)=3 \\ \emptyset & \text{ otherwise,} 
\end{cases}
\ \text{and} \ 
\D_{\K,G}(t)=\begin{cases} 
t+1 & \text{ if } \car(\K)=3 \\ 1 & \text{ otherwise.} 
\end{cases}
$$
\end{example}

As in \S\ref{sec-charvar}, another fundamental approach to these invariants takes place when $G$ is considered as the fundamental group of $X$, a 
connected CW-complex of finite type (which we can assume has a single 0-cell $e_0$). In this case $M_G^A$ is
nothing but $H_1(X_K; \Z)$ where $X_K$ is the covering of $X$ associated with the subgroup $K=\ker (G\to A)$.
In this scenario, we can define new invariants; let $\tilde e_0$ be the preimage of the $0$-cell in $X_K$.

\begin{definition}
The \emph{Alexander module of $X$ with respect to $A$} is the relative homology $H_1(X_K,\tilde e_0; \Z)$
(as a $\Lambda^A_\Z$-module) and we will denote it by $\tilde M^A_G$.
Starting off with $\tilde M^A_G$ and analogously to the previous discussion, one can define characteristic
varieties ${\widetilde \Char}_k(\K,G)$ and multi-variable Alexander polynomials $\tilde \D^A_{\K,G}$ associated 
with $\tilde M^A_{\K,G}$.
\end{definition}

\begin{rem}
The relationship between $\tilde M^A_G$ and $M^A_G$ is clearly given by the exact sequence of the pair 
$(X_K,\tilde e_0)$, that is,
\begin{equation}
\label{eq-aug}
0 = H_1(\tilde e_0) \to M^A_G \to \tilde M^A_G \to \ker \left(H_0(\tilde e_0) \to H_0(X; \Z)\right) \to 0.
\end{equation}
Since $H_0(\tilde e_0)=\Z A e_0$, $H_0(X; \Z))=\Z e_0$, and the map is given by $a e_0 \mapsto 1$, its kernel
is nothing but the augmentation ideal 
$I_{\Z[A]}=\left\{\sum_{a\in A} n_a  \mid \sum_{a\in A} n_a=0\right\}$.
 \end{rem}

A free presentation of the Alexander module can be given from the following chain map using Fox derivation 
with respect to $A$:
\begin{equation}\label{eq-chaincomplex}
C_2^A(X) \otimes \Z[A] \ \rightmap{\tilde \delta_2}\ C_1^A(X) \otimes \Z[A] \to  \tilde M^A_G \to 0. 
\end{equation}
In order to describe the boundary map $\tilde \delta_2$, let us fix for any 2-cell $e_2 \in C_2(X)$ a certain 
closed path $\partial e_2$ representing its boundary as induced by the cell map on the boundary. It might
happen that such map is constant. In that case $\tilde{\delta}_2(e_2 \otimes a)=0$. Otherwise 
$\partial e_2$ can be written as a composition of closed paths (1-cells), say 
$\partial e_2=x_1^{\varepsilon_1} \cdot x_2^{\varepsilon_2} \cdot ... \cdot x_n^{\varepsilon_n}$,
where $\varepsilon_1=\pm 1$. Then $\tilde{\delta}_2$ can be described recursively as a function of its 
boundary $\tilde{\delta}_2(e_2 \otimes a)=D(\partial e_2,a)$ where 
\begin{equation}
\label{defD}
D( x^{\varepsilon} \cdot y,a)=
\left\{ \array{ll}
x \otimes a + D \left(y,\psi(x)a \right) 
& {\rm if \ } \varepsilon =1 \\
-x \otimes \psi^{-1}(x)a+D\left(y,\psi^{-1}(x)a \right) 
& {\rm if \ } \varepsilon =-1.
\endarray \right.
\end{equation}

\begin{rem}
The main computational advantage of $\tilde M^A_G$ over $M^A_G$ is the size of the representation matrices, see Remark~\ref{rem-fitting}. The 
number of rows (resp. columns) of the matrix $\Mat(\tilde \delta_2)$ is linear with respect to the number of generators 
(resp. relations) of $G$, whereas the number of rows (resp. columns) of the matrix $\Mat(\phi)$ described 
in~(\ref{mat-phi}) is quadratic (resp. cubic) with respect to the number of generators (resp. relations) of $G$.
\end{rem}

From~(\ref{eq-aug}) one can easily see that 
\begin{equation}
\label{charfox}
\Char_k(G) \setminus \{\one\}={\widetilde \Char}_{k+1}(G) \setminus \{\one\}.
\end{equation}
(see for instance~\cite{ji-thesis} for a proof).

Analogously, one can easily check that 
\begin{equation}
\label{alexfox}
\tilde \D^A_G = 
\begin{cases}
\D^A_G	& \text{ if } \rank A>1 \\
(t-1)^k \D^A_G & \text{ if } \rank A=1
\end{cases}
\end{equation}

\begin{rem}\label{rem-compare}
Let us compare these arguments with the chain complex introduced in Remark~\ref{rem-cell}. For $A=H$, $X_K=\tilde{X}$ and
$C_*(\tilde{X};\C)$ is 
\begin{equation*}
0\to C_2(\tilde{X};\C)\ \rightmap{\delta_2}\ C_1(\tilde{X};\C) \ \rightmap{\delta_1}\ C_0(\tilde{X};\C)\to 0;
\end{equation*}
the map $\delta_2$ is the complexification of $\tilde\delta_2$ in~\eqref{eq-chaincomplex}. For $\xi\in\Hom(G;\C^*)$, $\xi\neq\one$, $\C$ has a natural structure of $\Lambda_\C$-module denoted by $\C_\xi$. Given any $\Lambda_\C$-module~$V$
we can  produce a \emph{twisted} $\C$-vector space $V_\xi:=V\otimes_{\Lambda_\C}\C_\xi$.
We have the following properties:
\begin{enumerate}
\enet{(M\arabic{enumi})}
\item The complexified Alexander Invariant $M_{\C,G}$ is the homology in degree~$1$ of $C_*(\tilde{X};\C)$.
\item The complexified Alexander Module $\tilde{M}_{\C,G}$ is $\coker{\delta_2}$.
\item\label{rem-compare3} For $\xi\neq\one$, $\xi\in\Char^H_k M$ if and only if $\dim (M_{\C,G})_\xi\geq k$,
i.e., if and only if $\xi\in \Char_k(G)$. This follows from the fact that the operations \emph{taking homology} and $\otimes_{\Lambda_\C}\C_\xi$ commute. 
\end{enumerate}
\end{rem}

The systematic use of the Alexander polynomial to distinguish quasi-projective groups has been known
since Zariski~\cite{Zariski-irregularity}, even though technically the invariant was defined later
by Libgober~\cite{Li8}.

The following property will be strongly used in this paper as an obstruction to quasi-projective groups:

\begin{thm}[{\cite[Theorem 4.3]{DPS4}}]
\label{thm-suciu}
Let $G=\pi_1(M)\in \QP$ be the fundamental group of a smooth, connected, complex quasi-projective variety.
\begin{enumerate}
\enet{\rm($\A$\arabic{enumi})}
\item\label{thm-suciu1} 
If $b_1(G)=r \neq 2$, then the Alexander polynomial $\Delta_{G}$ has a single essential variable, that is, 
there exist a Laurent polynomial $P\in \Z[t^{\pm 1}]$ such that $\Delta_G(t_1,...,t_r))=P(t_1^{n_1}\cdots t_r^{n_r})$.
 \item \label{thm-suciu2} 
If $b_1(G) \geq 2$, and $\Delta_G$ has a single essential variable, then either 
   \begin{enumerate}
    \item 
      $\Delta_G = 0$, or
    \item 
      $\Delta_{\C,G}(t_1, \dots ,t_r)= P(u)$ in $\C[t_1^{\pm 1},...,t_r^{\pm 1}]$, where $P$ is a product of 
cyclotomic polynomials (possibly equal to 1), and $u = t^{n_1} \cdots t^{n_r}$, with $\gcd(n_1,\dots, n_r)=1$.
 \item If $G\in \PP$, then $\Delta_G=const$.
   \end{enumerate}
\end{enumerate}
\end{thm}

Theorem~\ref{thm-suciu}\ref{thm-suciu2} is a generalization of the single-variable result of Libgober~\cite{Li8}. 
A generalization of the Alexander polynomial
for the twisted case can be found in~\cite{ji-20100}.

\section{Artin-Tits groups}
\label{sec-artin}
Let $\G$ be a simplicial graph with vertices $V$, edges $E$ 
(an element of $E$ is a subset of $V$ with two elements)
and a labeling $\ell:E\to\N_{\ge 2}$ of the edges. 
The \emph{Artin-Tits group} $A_{\G}$ associated to $\G=(V,E,\ell)$ is given 
by the presentation:
\[
A_{\G}=\langle v\in V\mid \underbrace{uvu\dots}_{\ell(e)\text{ times }}=
\underbrace{vuv\dots}_{\ell(e)\text{ times }}\,\, 
\text{ if } e:=\{u,v\}\in E \rangle.
\]
Artin groups associated with the constant map $\ell=2$ are called \emph{right-angled} Artin groups and 
their graphs will be denoted simply by $\G=(V,E)$.

The \emph{Coxeter group} $W_{\G}$ associated with $\G=(V,E,\ell)$ is 
the quotient of $A_{\G}$ obtained by factoring out by the squares of the
generators:
\[
W_{\G}=A_{\G}/\langle v^2,\,v\in V\rangle.
\]

\begin{rem}
It is useful to think of the Artin-Tits groups and Coxeter groups
as particular instances of a more general class. Let $\G=(V,E)$ be 
a graph as above. To each vertex $v\in V$ we associate a group $G_v$. 
To each edge $e=\{u,v\}\in E$ we associate a group $G_e=G_{u,v}$ of the form
$G_e=G_u*G_v/(R_e)$, where $(R_e)$ is the subgroup of the free product 
$G_u*G_v$ normally generated by a set of words $R_e$. 
Then the {\em Pride group} $P_{\G}$ is defined as 
\[
P_{\G}=*_{v\in V}G_v/(R_e,\,e\in E).
\]
For example, in the case of $A_{\G}$ the vertex groups are all infinite
cyclic $G_v=\Z$ and the edge group of $e=uv$ labeled by $m$ is the 
two generators, one-relator group 
\[
G_e=\langle u,v \mid \underbrace{uvu\dots}_{m\text{ times }}=
\underbrace{vuv\dots}_{m\text{ times }}\rangle. 
\]
In the case of $W_{\G}$ the vertex groups are cyclic of order $2$,
$G_v=\Z_2$.
%whereas the edge groups coincide with those of $A_{\G}$.
Some of the facts that are proved here easily generalize to Pride groups.
\end{rem}

\begin{definition}
An Artin-Tits group $A_{\G}$ is called \emph{spherical} or of \emph{finite type}
if its Coxeter group $W_{\G}$ is finite. Otherwise $A_{\G}$ is of \emph{infinite type}.
\end{definition}

Among Artin-Tits groups of infinite type we distinguish certain classes.

\begin{definition}
The Artin-Tits group $A_{\G}$ is called \emph{euclidean} (or \emph{affine})
if its Coxeter group $W_{\G}$ is euclidean. If $W_{\G}$ is of infinite 
non-euclidean type then $A_{\G}$ is of \emph{general type}. 
Furthermore, $A_{\G}$ is called \emph{hyperbolic} if $W_{\G}$ is hyperbolic. 
\end{definition}

\begin{example}
Let us show some examples of Artin-Tits groups.

\begin{enumerate}
\item If $\G, \G'$ are two arbitrary graphs, then 
$A_{\G\sqcup\G'}=A_{\G}*A_{\G'}$, 
where  $\G\sqcup\G'$ is the disjoint union.
\item If $\G_1,\dots,\G_k$ are the connected components of
$\G$ then  $A_{\G}$ is the free product $A_{\G_1}*\dots*A_{\G_k}$.
\item $A_{\G*\G'}=A_{\G}\times A_{\G'}$,
where  $\G*\G'$ is the join graph with vertices $V\cup V'$, 
edges $E\cup E'\cup\{\{v,v'\}\mid v\in V, v\in V'\}$ with new edges
having labels $\mu (\{v,v'\})=2$.
\item If $\G$ is the graph with no edges and $n$ vertices, then $\A_{\G}=\F_n$, the free group on $n$ generators.
\item If $\G=K_n$, the complete graph on $n$ vertices, then the right-angled 
Artin group $\A_{\G}$ is $\Z^n$, the free abelian group of rank $n$. 
\item If $\G=K_{n_1,\dots,n_r}$, the complete multipartite graph 
on $n$ vertices partitioned into subsets of sizes $n_1,\dots,n_r$, 
then right-angled group $\A_{\G}$ is $\F_{n_1}\times\dots\times\F_{n_r}$.
\item If $\G$ is the complete graph $K_n$ with labeling 
$\ell(\{i,j\})=2$ if $|i-j|\ge 2$, and $3$ otherwise, 
then $\A_{\G}=B_n$, the braid group on $n$ strings.
\item Let $\G$ be the one edge graph on two vertices with labeling $m\ge 2$.  
It is well known that $\A_{\G}=\langle u, v \mid uvu\dots=vuv\dots\rangle$ 
is isomorphic to the group $\pi_1(\C^2\setminus C)$ of the affine curve 
$C=\{x^2=y^m\}$ having an $\aA_{m-1}$ singular point at the origin. 
\item If $\G$ has two vertices and no edges, then the free group $\A_{\G}=\F_2$ is isomorphic to 
$\pi_1(\C^2\setminus C)$ for $C$ the union of two parallel lines in $\C^2$. 
\end{enumerate}
\end{example}

\begin{rem}
Thus $A_{\G}\in\QP$ for $\G$ a graph on two vertices. 
\end{rem}

\begin{example}
Another interesting case is that of graphs with three vertices.
All such graphs will be described by graphs $\G(p,q,r)$ 
as in Figure~\ref{fig-triangle}, with $p\geq q\geq r$, where 
$p,q,r\in \N\cup \{\infty\}$ means that there is no edge.
These groups will be denoted by $A(p,q,r)$.
\begin{enumerate}
\item If $\G$ has no edges, then $A_{\G}\in\H\subset\QP$
since it is the group of the complement of three parallel lines in~$\C^2$.

\item The case $\G(\infty,\infty,k)$ can also be disregarded as follows. In fact,
$\Char_1(A(\infty,\infty,k))=\T_{A_\G}$, whereas $\Char_2(A_\G)$ contains irreducible 
components of proper dimension, that is, positive but stricly less than 
the maximal. This contradicts Proposition~\ref{prop-obs1}\ref{prop-obs1-4}.

\item For the case where $\G$ has two edges there are partial results in 
Remark~\ref{rem-2edges} and Example~\ref{ex-2edges}.

\item For the general case $p,q,r\in \N$ there are partial results in 
Theorem~\ref{thm-triangle} and 
Examples~\ref{ex-333}-\ref{ex-spherical}, \ref{ex-666}, and~\ref{ex-542}.
\end{enumerate}
\end{example}

\begin{figure}[hb]
\centering
\includegraphics[scale=.5]{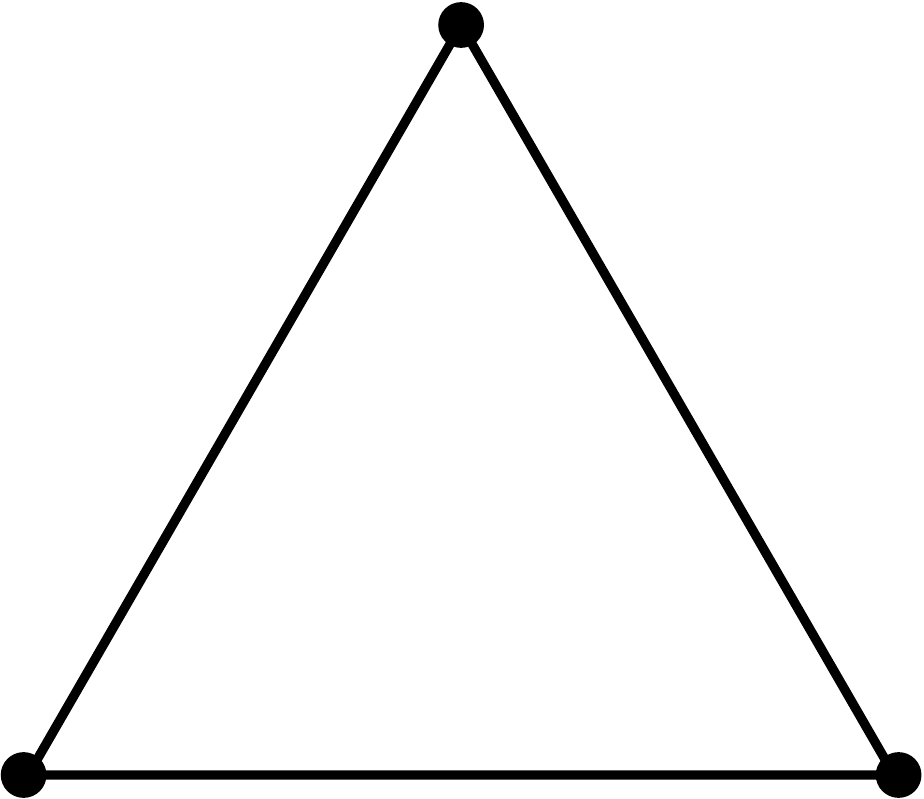}
\begin{picture}(0,0)
\put(0,0){3}
\put(-93,62){$r$}
\put(-74,116){1}
\put(-74,10){$p$}
\put(-148,0){2}
\put(-53,62){$q$}
\end{picture}
\caption{$\G(p,q,r)$}
\label{fig-triangle}
\end{figure}

We are interested in the following general question (see also~\cite{DPS0}):

\begin{problem}
\label{probl:artqproj}
Which Artin groups $A_{\G}$ are (quasi-)projective?
\end{problem}

This problem is partially solved for quasi-projective right-angled Artin groups as follows.

\begin{thm}[{\cite{DPS0,DPS1}}]\label{thm-ra}
The quasi-projective right-angled Artin groups $A_{\G}$ are
precisely those associated to a complete multipartite graph $\G=K_{n_1,\dots,n_r}$. 
In particular, $A_{\G}\in\QP$ if and
only if $A_{\G}=\F_{n_1}\times\dots\times\F_{n_r}$.
\end{thm}

As a consequence of classical invariant theory for spherical and euclidean
Coxeter groups we have the following result.

\begin{thm}
\label{thm:deftype1}
All spherical and euclidean Artin-Tits groups $A_{\G}$ are hypersuface
groups.
\end{thm}

\begin{proof}
It is enough to prove it for irreducible Artin-Tits groups.
We first deal with the spherical case. 

We will follow E.~Brieskorn~\cite{Br}.
Let $W$ be a finite Coxeter group. Realize $W$ as a reflection group in 
$\GL(V_{\R})$, for a real vector space $V_{\R}$ of dimension $n$. 
Let $V=V_{\R}\otimes\C$ be its complexification.
We have a proper action of $W$ on $V$ with quotient map $\pi:V\to V/W$ a branched
cover with branch locus $\D$, the hypersurface defined by the union of all reflecting 
hyperplanes of $W$ (we abuse the notation by using the same symbol for a hypersurface 
and a defining polynomial). The action of $W$ is free on $V\setminus \D$, and the 
image of the $W$-invariant polynomial $D^2$ is the discriminant locus $\D=\pi(D^2)$ 
of $\pi$. 

Thus the restriction $\pi:V\setminus \D\to V/W\setminus D$ is an unbranched
cover with group $W$. In fact, $V/W$ is isomorphic to the complex affine space $\C^n$
as provided by the Chevalley's classical invariant theory result: $\C[V]^W$ is a 
polynomial algebra $\C[f_1,\dots,f_n]$, and the isomorphism $V/W\to\C^n$ is given by 
$[v]\to(f_1(v),\dots,f_n(v))$. The fundamental group $\pi_1(\C^n\setminus D)$ of the 
discriminant complement is nothing but the Artin-Tits group $A_{W}$ associated to $W$.

We now consider the eculidean case. We will follow N.~Bourbaki~\cite{Bo}.
Let $R$ be an irreducible root system of rank $n$ in a real vector space $V_{\R}$,
with root lattice $Q$ and Weyl group group $W$. The affine Weyl group 
$\widetilde{W}=Q\rtimes W$ acts on $V=V_{\R}\otimes\C$ by affine reflections.  
Let $T$ for be complex torus with rational character lattice the weight lattice $P$ of $R$. 
For $\l\in P$ write $e^{\l}$ for the corresponding character in $T$. The exponential map 
$e:V\to T$ determines a short exact sequence $0\to Q\to V\to T\to 1$ that
provides an identification of the orbit spaces $V/\widetilde{W}=T/W$.
Classical exponential invariant theory tells us that the invariants $\C[T]^W$ of the 
algebra $\C[T]$ of Fourier polynomials on $T$, under the natural $W$-action is a 
polynomial algebra $\C[z_1,\dots,z_n]$. Hence the varieties $V/\widetilde{W}=T/W$
are in fact isomorphic to the affine space $\C^n$. If $\l_1,\dots,\l_n$ are
the fundamental weights in $P$ then we may take $z_j=\sum_{\mu\in W\cdot\l_j}e^{\mu}$,
where $W\cdot\l$ is the $W$-orbit of $\l$. In this way we obtain a branched
cover $\tilde\pi:T\to T/W$ with branch locus the Weyl denominator 
$\D=\prod_{\a\in R_{>0}}(e^{\a/2}-e^{-\a/2})$ and discriminant locus 
$D=\tilde\pi(\D^2)$. In this way we obtain an unbranched cover 
$V\setminus\widetilde{\D}\to V/\widetilde{W}\setminus \widetilde{D}=T/W\setminus D$,
where $\widetilde{\D}$ is the union of all reflecting hyperplanes of $\widetilde{W}$.
Recall that the affine Weyl groups are precisely the euclidean Coxeter groups.
Now according to~\cite{Du} the euclidean Artin-Tits group $A_{\widetilde{W}}$ is 
isomorphic to the the fundamental group 
$\pi_1(V/\widetilde{W}\setminus \widetilde{D})=\pi_1(T/W\setminus D)=\pi_1(\C^n\setminus D)$.
\end{proof}

The procedures described in the proof of Theorem~\ref{thm:deftype1} lead to concrete
equations for the hypersufaces whose complement have the Artin-Tits groups as fundamental 
groups. This has been carried out by E.M.~Opdam~\cite{Op} in the euclidean rank $2$ case. 
We will just give the results here.

\begin{example} 
\label{ex-333}
The type $\widetilde{\aA}_2$ corresponds to the Artin-Tits group $A(3,3,3)$.
This is the complement in $\C^2$ of the quartic 
$D=\{z_1^2z_2^2-4z_1^3-4z_2^3+18z_1z_2-27=0\}$. This curve is the tricuspidal quartic
where the line at infinity is the bitangent line.
\end{example}

\begin{example} 
The type $\widetilde{\C}_2$ corresponds to the Artin-Tits group $A(4,4,2)$.
This is the complement in $\C^2$ of the reducible quartic 
$D=\{(z_1^2-4z_2)(2z_1+z_2+4)(-2z_1+z_2+4)=0\}$
having $1$ node and $2$ tacnodes as singularities: a parabola with two tangent lines intersecting outside the conic.
\end{example}

\begin{example} 
The type $\widetilde{\gG}_2$ corresponds to the Artin-Tits group $A(6,3,2)$.
This is the complement in $\C^2$ of the reducible quintic 
$D=\{(z_1^2-4z_2)(-4z_1^3+z_2^2+12z_1z_2+24z_2+36z_1+36)=0\}$; this is
a cuspidal cubic (where the line at infinity is the tangent at the inflection point)
and a parabola (with the same point at infinity) intersecting at two smooth points,
with intersection multiplicities~$1$ and~$3$. The curve $D$ has singularities of type $\aA_1$, $\aA_2$, and $\aA_5$.
\end{example}

\begin{example}
\label{ex-spherical}
The spherical Artin triangle groups $A(p,q,r)$, $\frac{1}{p}+\frac{1}{q}+\frac{1}{r}>1$ can also be obtained in
an easy way by curves having singularities of type $\aA_{p+1},\aA_{q+1},\aA_{r+1}$ as follows:
\begin{description}
\item[$A(n,2,2)$] The curve of equation $(y^2-x^n)(x-1)=0$.
\item[$A(3,3,2)$] The curve of equation $8 y^3+3 y^2-6 x^2 y-x^2 (2+x^2)=0$. 
The curve has two ordinary cusps $\aA_2$ and one ordinary node $\aA_1$; there is only one branch at infinity,
which is a non-ordinary inflection point.
\item[$A(4,3,2)$] The curve of equation $(y^2-x^3)(2 y-3 x+1)=0$ the line being an ordinary tangent to the cubic.
\item[$A(5,3,2)$] The curve of equation $3125 y^3+16(16-125 x) y^2-4 x^2 (32-225 x) y+4 x^4 (4- 27 x)=0$.
It is a rational affine curve with three singular points of type $\aA_1$, $\aA_2$, and $\aA_4$. 
There is only one place at infinity; at this place the projectivized curve has a singular point of type $\aA_4$ such that
the tangent line (the line at infinity) has intersection number~$5$.
\end{description}
\end{example}

\begin{rem}
All the curves presented in the examples above have a common structure. In suitable coordinates
they can be presented as in Figure~\ref{fig-curva}, where the \emph{vertices} are singular double points. Note that for odd indices
the picture should continue till infinity. These curves are defined by a braid monodromy; if we choose a vertical line
close to the leftmost singular points, it is defined by four braids in four strings: $\sigma_1^{p+1}\cdot\sigma_3^{q+1}$, $\sigma_2$ and $\sigma_1^{-1}\cdot\sigma_3^{-1}\cdot\sigma_2^{r+1}\cdot\sigma_1\cdot\sigma_3$.
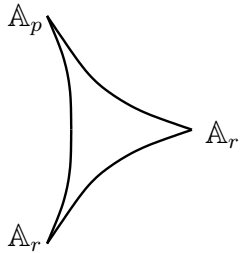
\begin{figure}
\centering
\includegraphics[scale=.5]{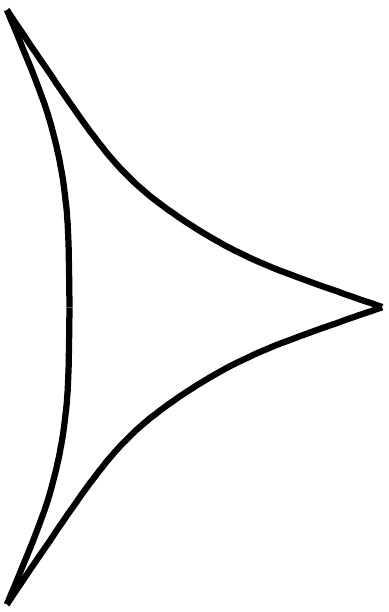}
\begin{picture}(0,0)
\put(0,39){$\aA_r$}
\put(-74,83){$\aA_p$}
\put(-74,0){$\aA_r$}
\end{picture}
\caption{Triangle curve}
\label{fig-curva}
\end{figure}
Following ideas of L.~Rudolph~\cite{Ru2}, S.~Orevkok~\cite{Or} proved that there exists 
an analytic curve in $\D\times\C$ ($\Delta$ a disk) realizing such a braid monodromy. 
For the elliptic and euclidean case the above Examples show that these curves can be 
extended to algebraic curves in $\C^2$ such that the complements in $\D\times\C$ and 
$\C^2$. One can construct for small values of $p,q,r$ (in the hyperbolic case) 
algebraic curves realizing this braid monodromy over the disk, but the monodromy of the 
curves is not trivial outside the disk. In fact, any attempt to construct hyperbolic 
triangle curves in $\C^2$ violates the Riemann-Hurwitz principle. On the other side, using 
ideas of L.~Rudolph in~\cite{Ru1}, any hyperbolic triangle curve is the fundamental 
group of a Stein surface properly embedded in $\C^4$.
\end{rem}

Now we turn to the question of projectivity of the Artin groups.

\begin{thm}
\label{thm:deftype2}
All spherical and euclidean Artin-Tits groups $A_{\G}$ are not projective.
\end{thm}

\begin{proof}
The abelianization of $A_{\G}$ is calculated for example in~\cite{De}.
Following~\cite[\S11.9]{DPS0}, let $\G_\text{\rm odd}$ be the graph obtained from $\G$ 
by removing all even labeled edges. Then $A_{\G}/A_{\G}'$ is isomorphic to $\Z^m$, 
where $m$ is the number of connected components of $\G_\text{\rm odd}$. 
Thus, if $\G_\text{\rm odd}$ is connected the abelianization of $A_{\G}$ is $\Z$.

Recall that the simplest necessary condition for group to be projective is 
to have even first Betti number. If we go over the list of irreducible 
spherical and euclidean Artin groups $A_{\G}$, we quickly discover that: 
$b_1(A_{\G})=1$ for the types $\aA_n, \dD_n, \E_k, \hH_l, \I_2(m)$, with $m$ odd,
$\widetilde{\aA}_n, \widetilde{\dD}_n, \widetilde{\E}_k$, $b_1(A_{\G})=1$ for
type $\widetilde{\C}_n$, and $b_1(A_{\G})=2$ for the rest of the cases.  

Now, for the types with small number of nodes $\F_4, \gG_2, \widetilde{\F}_4,\widetilde{\gG}_2$ 
and $\I_2(m)$, with $m$ even, one can easily exhibit low index subgroups of $A_{\G}$
that have and odd first Betti number. Although the types $\bB_n$ and $\widetilde{\bB}_n$
can also be dealt with that way, we take here a different road.

We are going to use an obstruction to projectivity (in fact to \emph{K{\"a}hler-ness}) obtained 
in~\cite{ABR}, following ideas of M.~Gromov in~\cite{Gr}, see also~\cite{ABCKT} for more background.
More precisely, if a group $G$ fits into an extension $1\to K\to G \to Q \to 1$ such that $K$ 
is finitely generated and $Q$ has an infinite number of ends $e(Q)$, then $G$ cannot be K{\"a}hler, 
hence projective. In fact, the group $G$ has also $e(G)=\infty$, which is not possible for a 
projective group (cf.~\cite{ABR} using the results in~\cite{Gr}). In~\cite{Ar2} this fact is used 
to show that pure braid groups $P_n, n\ge 3$ are not projective. Indeed we have a surjection 
$P_n\to P_3$ with finitely generated kernel, and $e(P_3)=\infty$ as $P_3=\F_2\times\Z$. 
We will follow the same line of reasoning.

First, it is known that the Artin group $A(B_n)$ surjects onto $A(A_{n-1})$ with kernel a 
free group of rank $n$, see~\cite{Cr, CP}. Then Arapura's discussion in~\cite{Ar2} gives 
that the braid group $A(A_{n-1})$ has infinitely many ends.

Secondly, it is proved by D.~Allcock~\cite{Al} that $\widetilde{\bB}_n, n\ge 3$ is an 
index $2$ subgroup of the braid group $B_n(\Sigma)$ of a $2$-dimensional orbifold
$\Sigma$, so it is enough to show that this group has infinitely many ends.  
The orbifold  $\C^*_\varphi$ where $\varphi(\frac{1}{2})=2$ and it is the only point with value
different from~$1$. 
Its associated pure braid group $P_n(\Sigma)$ is the kernel of the canonical 
surjection $B_n(\Sigma)\to S_n$ onto the symmetric group. As with the ordinary
pure braid groups we have a surjection $P_n(\Sigma)\to P_3(\Sigma)$ with
finitely generated kernel. It is readily seen that $P_3(\Sigma)$ has 
infinitely many ends.

In fact, the general irreducible case can be treated as above by showing that they have an 
infinite number of ends. This way one shows the result for the non-irreducible case, since
the product of two groups with infinite number of ends also has an infinite number of ends.
\end{proof}

\begin{conj} 
\label{conj:gentype}
All general type Artin groups $A_{\G}$ are not quasi-projective.
\end{conj}

This was already known for right-angled Artin groups (cf.~\cite{DPS0}).

\begin{definition}
A graph $\G=(V,E,\ell)$ is called \emph{even} if the labeling
$\ell$ takes only even values. In that case the Artin group 
$A_{\G}$ is called of \emph{even type}.
Suppose $\G=(V,E,\ell)$ is an arbitrary graph. Then $\G_{ev}$
is the graph obtained from $\G$ by identifying the endpoints
of each odd labeled edge and then removing the loops.
\end{definition}

\begin{rem}
Note that $\G_{ev}$ is not the same as the odd contraction of $\G$
considered in~\cite[\S11.9]{DPS0} since the even-labeled edges are
kept in $\G_{ev}$.

Since Theorem~\ref{thm-ra} solves the problem for right-angled groups,
we say that an Artin group is of \emph{strictly even type} if it is of 
even type and not right-angled.
\end{rem}

\begin{rem}
\label{rem:block}
Note that if $\G$ has connected components $\G_1,\dots,\G_r$
then the Alexander matrix $M_{\G}$ of $A_{\G}$ can be seen to be, 
for a suitable ordering of the vertices of $\G$, a block matrix 
with blocks $M_{\G_1},\dots,M_{\G_r}$.
\end{rem}

In~\cite[Theorem 9.9]{DPS0} it is proved that a free product $G_1*G_2$ 
of two $1$-formal groups $G_i$ with $b_1(G_i)>0, i=1,2$ and presented
by commutator relators is in $\QP$ if and only if both $G_i$ are free.
We apply this result to the case of even Artin groups.

\begin{prop} 
A free product $A_{\G_1}*A_{\G_2}$ of even Artin groups is not 
quasi-projective unless it is free. 
\end{prop}

Thus if we focus on the even graphs $\G$, we only need to consider 
quasi-projectivity of $A_{\G}$ for connected $\G$.

\begin{prop} 
\label{prop:e1ag}
For a connected graph $\G=(V,E,\ell)$ let $n_v$ be the valence of the
vertex~$v$. Then the Fitting ideal $F_1:=F_1(\tilde M_{A_{\G}})$ is of the following form:
\[
F_1=I^{\e}\cdot\left(\prod_{v\in T_{ev}}(t_v-1)^{n_v-1}\cdot
\prod_{e\in T}\D_{\ell(e)}(t_e),T\text{ spanning tree of }\G\right),
\]
where:
\begin{enumerate}
\enet{\rm(\arabic{enumi})}
\item $I$ is the augmentation ideal,
\item $\e=1$ if there exists an even spanning tree for $\G$, 
otherwise $\e=0$,
\item $t_e =t_u t_v$ if $\ell(e)$ is even, otherwise $t_e=t_u=t_v$,
\item $\D_{2 k}(t):=\dfrac{t^{k}-1}{t-1}$ and $\D_{2 k+1}(t):=\dfrac{t^{2k +1}+1}{t+1}$.
\end{enumerate}
\end{prop}

\begin{proof}
The proof is a pleasant exercise. 
The Alexander matrix $M=M_{\G}$ of $A_{\G}$ has $|E|$ rows and $|V|$ columns. 
Let $n=|V|$ and identify $V$ with $[n]$ in increasing order, and $E$ with 
a subset of $\binom{[n]}{2}$, ordered lexicographically, where $[n]:=\{1,\dots,n\}$. 
The entry of $M$
on row $jk$ and column $i$ is denoted as $m_{jk,i}$.
As $\G$ is connected, it has at least $n-1$ edges. Let $m=|E|-n+1$. 
The proof is carried out by double induction on $n\ge 1$ and $m\ge 0$. 

Suppose $m=0$. Thus $\G$ is a tree.
Then the Fitting ideal $F_1$ is generated by the
$n$~minors of $M_{\G}$ of codimension $1$. Let $M_i$ be the minor obtained
by removing the $i$-th column of $M$. Let $n_i$ be the valence
of the vertex $v=i$ and $ij_1,\dots,ij_{n_i}$ the edges at $i$.
Then the entry $m_{ij_p,i}$ is up to sign of the form 
$(t_i-1)^{((-1)^{\ell(ij_p)}+1)/2}\cdot\D_{ij_p}$. Moreover $m_{ij_p,i}$ 
is the only non-zero entry on the row $ij_p$ of the minor $M_i$.
Expanding $M_i$ along the rows $ij_p, \, 1\le p\le n_i$, and using 
Remark~\ref{rem:block}, we obtain
\[
M_i=\prod_{p=1,n_i}(t_i-1)^{((-1)^{\ell(ij_p)}+1)/2}\cdot\D_{ij_p}\cdot
\prod_{p=1,n_i} M_{j_p}(T_p),
\]
where $T_p$ is the sub-tree of $\G$ not containing vertex $i$
and edges $ij_1,\dots,ij_{n_i}$, and $M_{j_p}(T_p)$ is the codimension $1$
minor of the Alexander matrix of $A_{T_p}$ obtained by removing its
column corresponding to vertex $j_p$.

The last formula, combined with the induction hypothesis on $n$ 
finishes the proof for $m=0$. If $m\ge 1$, we proceed as follows.

Let $E'\subset E$ be a set of rows of $M$ of size $n-1$, and denote by 
$\G'$ the full subgraph of $\G$ on vertices $V$ and edges $E'$.
Let $M_{E',i}$ be the codim $1$ minor of $M$ obtained by removing
its $i$-th column and the rows in $E\setminus E'$.  
Then $M_{E',i}$ is the minor $M'_i$ of $M'=M(A_{\G'})$.
By Remark~\ref{rem:block}, $M'$ can be seen as a block matrix
with blocks $M_{\G'_1},\dots,M_{\G'_r}$, where $\G'_1,\dots,\G'_r$
are the connected components of $\G'$. If $r=1$ then $\G'$ is a 
spanning tree for $\G$, and we are done. If $r>1$ then $M'_i=0$.
\end{proof}

As an immediate corollary we have the following:

\begin{cor} 
For a connected tree $T$ the Alexander polynomial $\D(A_{\G})$ of $A_{\G}$
is given by:
\[
\D(A_{\G})=\prod_{v\in T_{ev}}(t_v-1)^{n_v-1}\cdot\prod_{e\in T}\D_{\ell(e)}(t_e).
\]
\end{cor}

\begin{prop}\label{prop-tree}
If $\G$ is an even tree with at least $3$ vertices, then $A_{\G}\not\in\QP$.
\end{prop}

\begin{proof}
Its Alexander polynomial $\D(A_\G)$ has at least $2$ essential variables and hence
Theorem~\ref{thm-suciu}\ref{thm-suciu1} applies.
\end{proof}
\begin{rem}
\label{rem-2edges}
In particular, the even graphs with three vertices and two edges are not in~$\QP$.
\end{rem}

\begin{prop}\label{prop-cycle}
If $\G$ is a strictly even $n$-cycle $\C_n$, $n\ge 4$, then $A_{\G}$ is not quasi-projective.
\end{prop}

\begin{proof}
The characteristic variety $\VV_1(A_{\G})$ does not pass the
quasi-projectivity test. Let $2 m_1,\dots,2 m_n$ be the weights.

If $n\geq 5$, then $V_{i,i+1}=\{t_i=t_{i+1}=1\}$ is an irreducible component.
In particular,  $V_{1 2}$ and $V_{2 3}$ intersect in a codimension~$3$ variety
contradicting Proposition~\ref{prop-obs1}\ref{prop-obs1-1}.

For $n=4$ we may assume that $m_1>1$ and we proceed in a similar way.
\end{proof}

\begin{thm} 
\label{thm:evnc}
If $\G$ is a strictly even non-complete graph $n\ge 3$, then $A_{\G}$ is not quasi-projective.
\end{thm}

\begin{proof}
The characteristic variety $\VV_1(A_{\G})$ does not pass the
quasi-projectivity test. It contains two irreducible components of $\VV_1(A_{\G})$ intersecting in a
subvariety of dimension greater than zero. 
Since  $\G$ not complete, there exists a subset $W\subset V$ that
disconnects $V$. Then from Proposition~\ref{prop:e1ag} it 
follows that both 
\[
C_W=\cap_{w\in W}\{t_w-1=0\},\, 
C_{W,w_0,v}=\cap_{w\in W\setminus\{w_0\}}\{t_w-1=0\}\cap\{t_{w_0}t_v+1=0\},
\]
are components of $\VV_1(A_{\G})$, where $v\in W$ is any neighbor of 
$w_0$ in $\G$. Clearly, their intersection 
$C_W\cap C_{W,w_0,v}=\cap_{w\in W}\{t_w-1=0\}\cap\{t_v+1=0\}$ 
is of strictly positive dimension since $n\ge 3$. 
By Proposition~\ref{prop-obs1}\ref{prop-obs1-1} we are done.
\end{proof}

Theorem~\ref{thm:evnc} leaves the case of a strictly even complete graph 
$\G=K_n,\, n\ge 3$ open. Nevertheless one can still use properties of the 
characteristic varieties to rule out quasi-projectivity for such graphs
in many instances. For simplicity, we consider here just the case $\G=K_3$. 

\begin{thm}
\label{thm-triangle}
The hyperbolic groups $A(2 p, 2 q,2 r)$ are not quasi-projective if 
$(p, q, r)\neq (k,k,k)$ with $k$ odd.
\end{thm}

\begin{proof} 
From the Alexander matrix of $G=A(2p,2q,2r)$ we readily obtain the Fitting ideals $F_i$ of $\tilde M_G$:
\begin{align*}
F_1= & I\cdot\left((t_1-1)\D_r(t_1t_2)\D_q(t_1t_3),(t_2-1)\D_r(t_1t_2)\D_p(t_2t_3),(t_3-1)\D_q(t_1t_3)\D_p(t_2t_3)\right),\\
F_2= & \left((t_1-1)\D_r(t_1t_2),(t_2-1)\D_r(t_1t_2),(t_2-1)\D_p(t_2t_3),(t_3-1)\D_p(t_2t_3),\right.\\
& \left. (t_1-1)\D_q(t_1t_3),(t_3-1)\D_q(t_1t_3)\right),
\end{align*}
see Figure~\ref{fig-triangle} and Proposition~\ref{prop:e1ag} for the notation.

We may restrict to $G=A(2p,2q,2r),\, p\ge 2,q\ge 2,r\ge 1$ since $A(2p,2,2),\,p\ge 1$ 
are spherical and $A(4,4,2)$ is euclidean. There are two cases $r=1$ and $r\ge 2$.

\begin{notation}
Denote by $\mu_N^*$ the set of the $N$-th roots of unity distinct from $1$.  
\end{notation}

\begin{case}
Suppose $r=1$, $p\ge 3$, and $p, q$ are coprime.
\end{case}

Then $\VV_1(G)$ has $pq$ irreducible components,
one of dimension zero $(1,1,1)$, and $pq-1$ of dimension one:
\begin{align*}
C_{1,\zeta}=&\{t_1-1=t_2t_3-\zeta=0\},\zeta\in\mu_q^*\\
C_{2,\zeta,\xi}=&\{t_2t_3-\zeta=t_1t_2-\xi=0\},\zeta\in\mu_q^*,\xi\in\mu_p^*,\\
C_{3,\xi}=&\{t_3-1=t_1t_2-\xi=0\},\xi\in\mu_p^*.
\end{align*}
Furthermore $\VV_2(G)$ consists of just $(1,1,1)$.
Now the $1$-dimensional components of $\VV_1(G)$ intersect as follows: 
\[
C_{1,\zeta}\cap C_{2,\zeta,\xi}=\{(1,\xi,\zeta\xi^{-1})\},\,
C_{3,\xi}\cap C_{2,\zeta,\xi}=\{(\xi\zeta^{-1},\zeta,1)\}.
\]
Clearly $(1,\xi,\zeta\xi^{-1})$ and $(\xi\zeta^{-1},\zeta,1)\}$ 
do not belong to $\VV_2(G)$. This case
follows from  Proposition~\ref{prop-suma}.

\begin{case}
Suppose $r=1$,  $p\ge 3$ and $\gcd(p, q)=d>1$.
\end{case}

The types of components
of $\VV_1(G)$ are identical to the ones when $p, q$ are coprime.
Furthermore $\VV_2(G)$ consists of $p+q-1$ points: $(1,1,1),(1,\xi,1)$, 
with $\xi\in\mu_p^*\cap\mu_q^*$. Nevertheless, some of the intersections 
\[
C_{1,\zeta}\cap C_{2,\zeta,\xi}=\{(1,\xi,\zeta\xi^{-1})\}
\text{ and }
C_{3,\xi}\cap C_{2,\zeta,\xi}=\{(\xi\zeta^{-1},\zeta,1)\}
\]
will not be in $\VV_2(G)$,
unless $\xi=\zeta\in\mu_p^*\cap\mu_q^*$. If $p$ and $q$ are distinct then one can always 
find a $\xi\in\mu_p^*$ and a $\zeta\in\mu_q^*$ that are not equal. 
If $p=q$ then $\zeta\xi^{-1}$ must be $1$ for all pairs $\xi,\zeta\in\mu_p^*$,
which is only possible when $p=q=2$ (excluded). 

Thus, there will always be a pair of components intersecting outside $\VV_2(G)$
and thus the group is not in $\QP$, by Proposition~\ref{prop-suma}.

\begin{case}
Suppose $p,q, r\ge 2$, not all equal.  
\end{case}
Let $\a\in\mu_p^*,\b\in\mu_q^*,\g\in\mu_r^*$.
Then $\VV_1(G)$ has $(1,1,1)$ as an isolated point, and 
$pq+qr+rp-p-q-r$ components of dimension one, of two types:
\[
C_{1,\b}=\{t_1-1=t_2t_3-\b=0\} \text{ and } C_{1,\a,\g}=\{t_1t_2-\a=t_1t_3-\g=0\},
\]
and similarly $C_{2,\g}, C_{3,\a}$, respectively $C_{2,\a,\b}, C_{3,\b,\g}$.
Furthermore $\VV_2(G)$ consists of $(1,1,1)$ and possibly points of the 
following three types: 
\begin{itemize}
\item having two trivial coordinates, namely $(\xi,1,1)$ with $\xi\in\mu_p^*\cap\mu_r^*$, 
$(1,\xi,1),\xi\in\mu_p^*\cap\mu_q^*$ and $(1,1,\xi),\xi\in\mu_q^*\cap\mu_r^*$);
\item having one trivial coordinate, namely $(1,\xi,\zeta)$ with $\xi\in\mu_p^*,\zeta\in\mu_r^*$
and $\xi\zeta\in\mu_q^*$ and the corresponding ones with $1$ in the other variables;
\item having no trivial coordinate $(t^{-1},t\xi,t\zeta)$ with $\xi\in\mu_p^*,\zeta\in\mu_r^*$
and $t^2\xi\zeta\in\mu_q^*$.
\end{itemize}

Now consider the following intersections of $1$-dimensional components of $\VV_1(G)$: 
\[
C_{1,\b}\cap C_{2,\a,\b}=\{(1,\a,\b\a^{-1})\},\,
C_{1,\b}\cap C_{3,\b,\g}=\{(1,\g\b^{-1},\b)\}. 
\]
In order for $(1,\a,\b\a^{-1})$ to be in $\VV_2(G)$ either 
$\a=\b\in\mu_p^*\cap\mu_q^*$, or $\b\a^{-1}\in\mu_r^*$.
Similarly $(1,\g\b^{-1},\b)\in\VV_2(G)$ if either 
$\b=\g\in\mu_q^*\cap\mu_r^*$, or $\g\b^{-1}\in\mu_p^*$.

Since $p,q,r$ are not all equal, we have $p>2$; we choose
$\a$ a primitive root in $\mu_p^*$, and $\b\in\mu_q^*$, also primitive, distinct from $\a$.
This implies that  $\b\a^{-1}\in\mu_r^*$.
If $\a^2\neq\b$ then we must also have $\b\a^{-2}\in\mu_r^*$. It follows that $\a\in\mu_r^*$,
which is also the case if $\a^2=\b$. We deduce that 
which implies $p\mid r$, which leads to a contradiction. 
Using~Proposition~\ref{prop-suma}, we deduce that the groups in this Case are not in $\QP$.

\begin{case}
\label{caso4}
Suppose $p=q=r\geq 2$, $p=2k$ even.
\end{case}

We show that $A(2p,2p,2p), p\ge 2$ is not in $\QP$,
by testing the characteristic varieties of its index $2$ subgroup $N$
defined by the homomorphism $x_1,x_2\to 0,x_3\to 1\mod 2$.
Then $N$ is generated by $x_i,y_i:=x_3x_ix_3^{-1},i=1,2$, and $z:=x_3^2$.
The relations are:
\[
R=\{
(x_1x_2)^{2k}=(x_2x_1)^{2k}, (y_1y_2)^{2k}=(y_2y_1)^{2k}, 
(x_iy_iz)^{k}=(y_izx_i)^{k}=(zx_iy_i)^{k}
\}.
\]

By abuse of notation 
we denote the generators of $H_1(N)=\Z^5$ by $x_1,x_2,y_1,y_2,z$.
We see that certain irreducible components of $\VV_1(N)$ intersect in dimension $1$.
More precisely, for $\xi\in\mu_k^*$ the following $2$-dimensional tori
\[
C_{1,2}=\{x_1-1=y_2-1=z-1=0\},\,C_{2,1}=\{x_2-1=y_1-1=z-1=0\},
\]
\[
C_{\xi,1}=\{x_1x_2-\xi=y_1-1=z-1=0\},\,C_{\xi,2}=\{x_1x_2-\xi=y_2-1=z-1=0\},
\]
\[
C_{1,\xi}=\{y_1y_2-\xi=x_1-1=z-1=0\},\,C_{2,\xi}=\{y_1y_2-\xi=x_2-1=z-1=0\},
\]
are components of $\VV_1(N)$, and we have that 
\[
C_{1,2}\cap C_{1,\xi}=\{(1,t,\xi,1,1)\mid t\in \C^*\},\,
C_{1,2}\cap C_{\xi,2}=\{(1,\xi,t,1,1)\mid t\in \C^*\},
\]
\[
C_{2,1}\cap C_{2,\xi}=\{(1,t,\xi,1,1)\mid t\in \C^*\},\,
C_{2,1}\cap C_{\xi,2}=\{(\xi,1,1,t,1)\mid t\in \C^*\}.
\]
Using Proposition~\ref{prop-obs1}\ref{prop-obs1-1}, we conclude
that these groups are not quasi-projective.
\end{proof} 

\begin{example}
\label{ex-666}
The case excluded in Theorem~\ref{thm-triangle} cannot be treated as Case~\ref{caso4}.
We sketch a proof for the simplest case $A(6,6,6)$. Let $N=\ker \psi$, where $\psi$ is 
the morphism sending each generator to the same non-trivial element of $\Z_3$. One can
check that $b_1(N)=7$ and that $\VV_4(N)$ has three irreducible components of dimension~2
through~$\one$. The group $N$, and thus $A(6,6,6)$, is not quasi-projective by 
Proposition~\ref{prop-obs2}\ref{prop-obs2-1}.
\end{example}

\begin{example}
\label{ex-2edges}
We will consider the case where $\G$ has three vertices and two edges labeled $2,3$,
say $\G(\infty,3,2)$ (see Figure~\ref{fig-triangle}). The representation 
$A(\infty,3,2)\to \Sigma_3$ given by $1\mapsto (1,2)$, $2 \mapsto 1$, and 
$3\mapsto (2,3)$ has as kernel the Artin group of a bamboo with 4 vertices and labels 
$(2,4,2)$, which is not quasi-projective by Proposition~\ref{prop-tree}.

The case $\G(\infty,3,3)$ can be treated analogously. The representation 
$A(\infty,3,3)\to \Sigma_3$ given by $1\mapsto (1,2)$, $2 \mapsto (2,3)$, and 
$3\mapsto (2,3)$ has as kernel the Artin group of a triangle $\G(\infty,4,4)$, 
which is not quasi-projective by Proposition~\ref{prop-tree}.
\end{example}

\begin{example}
\label{ex-542}
We will prove that $A(5,4,2)\notin \QP$ (note that $b_1(A(5,4,2))=2$). 
Consider the following representation $A(5,4,2)\to \Sigma_5$ given by 
$1\mapsto (2,3)$, $2\mapsto (2,3)(4,5)$, and $3\mapsto (1,2)(3,4)$.
Its kernel $K$ has Betti number 4, and $\Char_1(K)$ has three irreducible components 
$V^2_1, V_2^1, V_3^1$ ($\dim V_i^j=j$). However $V_1^2\cap V_2^1\notin \Char_2(K)$, which 
contradicts Proposition~\ref{prop-suma}.

The case $A(5,5,4)$ can be treated simarly (note that $b_1(A(5,5,4))=1$). 
Consider the following representation $A(5,5,4)\to \Sigma_5$ given by 
$1\mapsto (2,3)(4,5)$, $2\mapsto (2,4)(3,5)$, and $3\mapsto (1,2)(3,4)$.
Its kernel $K$ has Betti number 3, and $\Char_1(K)$ has three 1-dimensional irreducible 
components $V_1, V_2, V_3$. However $V_1\cap V_2\notin \Char_2(K)$, which 
contradicts Proposition~\ref{prop-suma}.
\end{example}

\section{Pencil map construction}
\label{sec-pencilmaps}

We fix a quasi-projective group $G=\pi_1(M)$. We assume that $b_1(G)>1$ and 
that the first characteristic variety $\VV_1(G)\subset\T_G$ is not $0$-dimensional.
Then, according to Theorem~\ref{thm-orb}, each irreducible component $W$ of dimension 
$d>0$ of $\VV_1(G)$ is of the form $W=f^*(V)$, where $f:M\to C$ is an orbifold morphism
(see Definition~\ref{def-orb-morph}), $b_1(C)=d$, and $V\subset \T_C$ is an irreducible
component of $\Char_1(C)$. 
In light of these facts, we make the following construction.

Suppose that $f_i: M\to C_i$, $1\le i\le k$ are the pencils determined by the 
irreducible components of $\VV_1(G)$ passing through the trivial character $\one\in\T_G$. 
Define $F:M\to C_1\times\dots\times C_k$ to be the product map and let 
$\phi=F_{\#}$ be the induced homomorphism
\begin{equation*}
\phi:=\phi_1\times\dots\times\phi_k:G\to D=\pi_1\times\dots\times\pi_k, 
\end{equation*}
where $\pi_i:=\pi_1(C_i)$, and $\phi_i=(f_i)_{\#}:G\to\pi_i$.
Note that none of the factors of $D$ is abelian.

If $N:=\ker\p$ and $S:=\im\phi$, then we have exact sequences: 
$1\to N\to G\to S\to 1$ and $1\to T\to D\to A\to 1$, where $T=\bar{S}$
is the normal closure of $S$ in $D$ and $A:=D/T$. Note that, in 
the terminology of Bridson-Miller~\cite{BM}, $S$ is a subdirect product 
of $D$, due to the surjectivity of each $\phi_i$. Note that if a subgroup $H$
of a direct product of $\pi\times \s$ is such that $\pi\cap H$ is trivial, 
then $H$ is isomorphic to a subgroup of $\s$. In view of this we will 
assume that $S$ intersects each factor $\pi_i$ non-trivially, that is,
$S$ is a full subdirect product of $G$ (cf.~\cite{BM}).

We have the following lemma. The proof is immediate and will be omitted, 
see also~\cite[Proposition 1.2]{BM}.

\begin{lem} 
The group $A$ is free abelian of finite rank.
\end{lem}

Note that, since $S$ is finitely presented, it is also of type $FP_2(\Q)$, 
by Bridson-Howie-Miller-Short~\cite[Theorem~B]{BHMS1}.
It follows, for all $n\ge 1$, that $S$ is type $F_n$ if and only if is type $FP_n(\Q)$.
Moreover, again by~\cite[Theorem~D]{BHMS1}, the normal subgroup $T$ shares the 
finiteness properties of~$S$.

\begin{rem} 
Recall that if $\one\notin W$, then $\Sh W$ is also an irreducible component of 
$\VV_1(G)$, unless $C$ is supported either over $\C^*$ or over an elliptic curve.
The orbifold version of $\phi$ is obtained by considering all orbifold morphisms
of $M$ (and not only those whose associated components of $\VV_1(G)$ pass through
the trivial character). Note that in this case $A$ may have torsion.
\end{rem}

\begin{prop} 
If $\phi$ is injective, then either $G$ is finite index in $D$,
or $G$ is of infinite index in $D$ and not of type $FP_k(\Q)$.
\end{prop}

\begin{proof}
Suppose $G$ has infinite index in $D$. The injectivity of $\phi$ 
implies $G=S$, making $G$ a direct subproduct of $D$.
Now Bridson-Howie-Miller-Short~\cite[Theorem~C]{BHMS2}
ensures the existence of a finite index subgroup $S_0$ of $S=G$
which is not type $FP_k(\Q)$, and we are done.
\end{proof}

In the case where $\phi$ is injective and $G$ is finite index in $D$,
then $G$ is clearly of type $FP_k(\Q)$.

\begin{cor}
If $G$ is an infinite index subgroup of $D$, then $M$ is not a $K(G,1)$.
\end{cor}

\begin{rem} 
By F.~Catanese~\cite{Cat1}, $\phi_i$ has finitely generated kernel $N_i$ 
if and only if either $g(C_i)=0$, or $\ge 1$ and $f_i$ has no multiple fibers. 
Note that $N=N_1\cap\dots\cap N_k$.
\end{rem}

\begin{conj}
The group $G$ is of type $FP_k(\Q)$ if and only if either
$\phi$ is not injective or $G$ has a finite index subgroup
which is a direct product of at most $k$ smooth curve groups.
\end{conj}

\begin{rem} 
Note that $D$ is linear and residually finite (resp. nilpotent).
Thus, if $\phi$ is injective, then $G$ is linear and 
residually finite (resp. nilpotent) as well.
\end{rem}

In the following examples we will exhibit some hyperplane arrangement groups and
study their pencil maps. In the context of line arrangements, a construction similar 
to our pencil map was considered by Cohen-Falk-Randell in~\cite{CFR}, where analogous 
finiteness results and examples are given.

\begin{example} 
Let $M=\P^2\setminus\A$, with $\A$ the arrangement
of $6$ lines in $\P^2$ with coordinates $x,y,z$, given by 
$Q=xyz(x-y)(y-z)(z-x)$. There is a total of five pencil morphisms:
four of them coming from the triple points of $Q$ and one coming
from the pencil of conics through the triple points. The pencil
map $\phi:G\to (\F_2)^5$ is not injective. In fact, for each pencil
$\phi_i$, $G$ fits into a split exact sequence of the form 
$1\to N_i\to G\to \F_2\to 1$, where $N_i$ is a free group of rank~3
generated by meridians. Consider three meridians $a_1,a_2,a_3$ 
such that two of them are meridians of two lines intersecting at a 
double point. Then the commutator $[[a_1,a_2],a_3]$ is in fact 
a non-trivial element of~$N=N_1\cap N_2 \cap N_3$.
\end{example}

\begin{example} 
Let $M=\P^2\setminus\A$, where $\A$ is the arrangement of $6$ lines in $\P^2$ 
given by $Q=xyz(x-y)(y-z)(z-\l x)$, $\l\neq 0,1$. Then 
$\phi:G\to (\F_2)^3$ is injective and $G=\pi_1(M)$ is the Stallings
group, fitting into the exact sequence: $1\to G\to (\F_2)^3\to \Z\to 1$.
\end{example}

The fact that Stallings' group is realizable as an arrangement group has been
first discovered by the third author and A. Suciu in 2004, but no satisfactory
extension of this observation to other arrangements was achieved then. 
We can now give a generalization of the previous example, by realizing all the 
Stallings-Bieri groups as arrangement groups. Further connections between
Bestvina-Brady type groups and fundamental groups of hypersurface complements
will be pursued elsewhere.

\begin{prop} 
Let $M_k:=\P^{k-1}\setminus\A$ and $k>3$, where $\A$ is the hyperplane 
arrangement in $\P^{k-1}$ with coordinates $x_1,\dots,x_k$, given by 
\[
Q=x_1\cdots x_k(x_1-x_2)(x_2-x_3)\cdots(x_{k-2}-x_{k-1})(x_{k-1}-x_{k})(x_k-\lambda x_1),
\]
$\l\neq 0,1$. Then the group $G_k=\pi_1(M_k)$ has type $\text{F}_{k-1}$, but not $\text{F}_k$.
\end{prop}

\begin{proof}
Consider the following map 
$$
(\C\setminus \{0,1\})^k \ \rightmap{\psi}\ \C^*, 
\quad
(x_1,\dots,x_k)\mapsto x_1\cdots x_k.
$$
At the level of fundamental groups it induces Bieri's map $(\F_2)^k \to \Z$ (cf.~\cite{Bi1}).
It is easy to see that $\psi$ is a locally trivial fibration outside the fiber of~1. 
The generic fiber is homeomorphic to $M_k$. Since the critical locus is of real codimension 
greater than~2, the kernel of Bieri's map is $G_k$ and the result follows from Bieri's arguments.
\end{proof}

\end{document}